\documentclass[a4paper,10pt]{amsart}
\usepackage{amsmath,amsfonts,amssymb,mathrsfs,esint}
\usepackage{enumitem,verbatim}
\usepackage{graphicx,color}
\usepackage{hyperref}

\newtheorem{theorem}{Theorem}
\newtheorem{lemma}[theorem]{Lemma}

\newtheorem{claim}{Claim}
\theoremstyle{remark}
\newtheorem{remark}[theorem]{Remark}
\theoremstyle{definition}
\newtheorem{definition}[theorem]{Definition}

\newcommand{\R}{\ensuremath{\mathbb{R}}}

\begin{document}
\title[Traveling combustion fronts]{(Bounded) Traveling combustion fronts
with degenerate kinetics}
\author{Natha{\"e}l Alibaud}
\author{Gawtum Namah}
\address[Natha\"{e}l Alibaud and Gawtum Namah]{ Ecole Nationale Sup\'erieure
de M\'ecanique et des Microtechniques\\
26 Chemin de l'Epitaphe\\
and\\
Laboratoire de Math\'ematiques de Besan\c{c}on, UMR CNRS 6623, Universit\'e
de Bourgogne Franche-Comt\'e\\
16 route de Gray\\
25030 Besan\c{c}on cedex, France}
\email{nathael.alibaud@ens2m.fr, gawtum.namah@ens2m.fr}
\thanks{}
\date{\today }

\begin{abstract}
We consider the propagation of a flame front in a solid periodic medium. The
model is governed by a free boundary system in which the front's velocity
depends on the temperature via a kinetic rate which may degenerate. We show
the existence of travelling wave solutions which are bounded and global.
Previous results by the same authors (cf. \cite{IFB}) were obtained for
essentially positively lower bounded kinetics or eventually which have some
very weak degeneracy. Here we consider general degenerate kinetics,
including in  particular  those of Arrhenius type which are commonly used in
physics.
\end{abstract}

\maketitle



\section{Introduction}

In this paper, we consider a flame front in a solid periodic medium $\mathbb{%
R}_{x}\times \mathbb{R}_{y}$ where the fresh region is a hypograph~$\{x<\xi
(y,t)\}$ with a temperature $T=T(x,y,t)$. The evolution of $(\xi ,T)$ will
be governed by the free boundary system 
\begin{equation}
\begin{cases}
T_{t}-\triangle T=0, & \quad x<\xi (y,t), \, t>0, \\[1ex] 
\xi _{t}+R(y) K(T)\sqrt{1+\xi _{y}^{2}}=\frac{\xi _{yy}}{1+\xi _{y}^{2}},
& \quad x=\xi (y,t), \, t>0,%
\end{cases}
\label{free-bdry-sys}
\end{equation}%
subject to the boundary conditions 
\begin{equation}
\begin{cases}
\frac{\partial T}{\partial \nu }=V_{n}, & \quad x=\xi (y,t), \\[1ex] 
T(x,y,t)\rightarrow 0, & \quad \mbox{as $x\rightarrow -\infty$},%
\end{cases}
\label{bdry-cond}
\end{equation}%
where $\nu 
:=
(1,-\xi _{y})/\sqrt{1+\xi _{y}^{2}}$ is the outward unit
normal and $V_{n}$ is the normal velocity of the front. The second equation
of \eqref{free-bdry-sys} states that the front propagates with a normal
velocity $V_{n}$ given by 
\begin{equation*}
V_{n}=-R(y) K(T)-\kappa ,
\end{equation*}%
where $\kappa $ is the mean curvature and $R \, K$ is the forcing term. The
latter depends on the temperature through a \textit{kinetic}  which is generally of Arrhenius type \textit{i.e.} of the form 
\begin{equation}
K(T)=Ae^{-\frac{B}{T}}  \label{Arrhenius}
\end{equation}%
for some physical positive constants $A$ and $B.$ The heterogeneity is given
through the function $R$ which represents the combustion rate. It is
typically a periodic step function for a striated medium obtained from a
superposition of different materials, which is the most common situation.
The heterogeneity may as well come from $A$, $B$ or other intrinsic
parameters such as the diffusivity, etc. Here we will assume that all these
parameters are independent of the striations and normalized to one except
for $R=R(y).$ We opted for such a simple problem to shed light on the main
mathematical difficulties but extensions are possible. For more details and
references about the physical model, see \cite{BrFiNaSc93}, \cite{ChNa97} and 
\cite{IFB}.

\bigskip

In this work we will focus on the existence of traveling wave solutions
(TWS) to~\eqref{free-bdry-sys}-\eqref{bdry-cond} \textit{i.e.} fronts
having a constant profile $\psi =\psi (y)$ and moving with a constant speed $%
c>0$. This comes to looking for solutions of the form%
\begin{equation*}
\xi (y,t)=-c\,t+\psi (y) \mbox{ and } T(x,y,t)=u(x+c\,t,y).
\end{equation*}%
It is convenient to fix the front through the change of variable~$%
x+c\,t\mapsto x$ which leads to the problem of finding a triplet~$(c,\psi
,u) $ satisfying%
\begin{equation}
\begin{cases}
c\,u_{x}-\triangle u=0, & \quad x<\psi (y), \\[1ex] 
\frac{\partial u}{\partial \nu }=\frac{c}{\sqrt{1+\psi _{y}^{2}}}, & 
\quad x=\psi (y), \\[1ex] 
u(x,y)\rightarrow 0, & \quad \mbox{as~$x\rightarrow -\infty$,}%
\end{cases}
\label{Eqn_laplacien}
\end{equation}%
and%
\begin{equation}
-c+R(y)K(u)\sqrt{1+\psi _{y}^{2}}=\,\frac{\psi _{yy}}{1+\psi _{y}^{2}},\quad
x=\psi (y).  \label{eqn_front_laplacien}
\end{equation}%
This problem has been dealt in \cite{IFB} by the present authors for a
positively lower bounded combustion rate $R$ and for a positive kinetic $K$
which can eventually very weakly degenerate at zero. This means that $K$ may
vanish but not too fast. The worst situation that we were able to consider
is when 
\begin{equation*}
K(u)\backsim \frac{1}{\ln \frac{1}{u}}\text{ as }u\rightarrow 0.
\end{equation*}%
The existence of a TWS was then proved in \cite{IFB} where the speed $c>0$
and the profile $\psi =\psi (y)$ is globally defined and bounded. Our aim is
to generalize this result to much more degenerate kinetics including those
of Arrhenius type \eqref{Arrhenius}, which is what seems to be the most
realistic. More precisely, we just assume that there are constants $R_{m},$ $%
R_{M},$ and $K_{M}\ $such that
\begin{equation}
\begin{cases}
R:\mathbb{T}\rightarrow \mathbb{R}\text{ is measurable with }0<R_{m}\leq
R(y)\leq R_{M}\text{, and} \\
K:\mathbb{R}^{+}\rightarrow \mathbb{R}^{+}\text{ is continuous,
nondecreasing with }0<K(u)<K_{M},%
\end{cases}
\label{HypsurRK}
\end{equation}%
where $\mathbb{T}$ is the torus $\mathbb{R}/\mathbb{Z}.$ Let us now define a
solution of \eqref{Eqn_laplacien} in the variationnal sense.

\begin{definition}
Let $\Omega 
:=
\{(x,y); \, x<\psi (y), \, y\in \mathbb{T}\}.$ Then given $c\in 
\mathbb{R}$ and  $\psi \in W^{1,\infty}(\mathbb{T}),$  a function $u$ will be called a
variationnal solution of \eqref{Eqn_laplacien} if $u\in H^{1}(\Omega )$
and satisfies%
\begin{equation}
c\int_{\Omega }u_{x}w+\int_{\Omega }\nabla u\nabla w=\int\limits_{x=\psi (y)}%
\frac{c \, w}{\sqrt{1+\psi _{y}^{2}}} \text{ for all } w\in H^{1}(\Omega ).
\label{weaksolution}
\end{equation}
\end{definition}

\noindent We now state our main result.

\begin{theorem}\label{mainthm}
Under \eqref{HypsurRK}, there exists a TWS $(c,\psi ,u)$ of \eqref%
{free-bdry-sys}-\eqref{bdry-cond} with $c>0,\psi \in W^{2,\infty }(%
\mathbb{T})$ and $0\leq u\in H^{1}(\Omega ).$ More precisely $\psi $
satisfies \eqref{eqn_front_laplacien} {\it a.e.} and $u$ is a  variationnal
solution of \eqref{Eqn_laplacien}. 
\end{theorem}

Technically speaking, it is not obvious that the front remains bounded since
the forcing term can \textit{a priori} vanish as $u$ goes to zero. It is
indeed well known that for pure geometric propagations of the form 
\begin{equation*}
-c+
H
(y)\sqrt{1+\psi _{y}^{2}}=\frac{\psi _{yy}}{1+\psi _{y}^{2}},\quad
x=\psi (y)
\end{equation*}%
with  $H$  not positively lower bounded, one can as well end up with unbounded
traveling fronts which in addition may not be globally defined. For more
details see \cite{Italiens}, where the existence of so-called \textit{%
generalized TWS} are discussed by the use of variational techniques. Here,
we not only provide a new existence result of  TWS  for \eqref%
{free-bdry-sys}-\eqref{bdry-cond} with general degenerate kinetics but
we also rigorously show that the front is globally defined and bounded.

\bigskip

The rest of the paper is organized as follows. To prove Theorem \ref{mainthm} we will
need three key lemmas stated and proved in Sections \ref{section2}, \ref{section3} and \ref{section4} respectively.
In these preliminary results, we will establish an \textit{a priori}
positive lower bound for the speed $c$ (Lemma \ref{lemma1}) and a lower bound for the
front's temperature (Lemma \ref{lemma5}) based on some adequate monotonicity property
of the temperature (Lemma \ref{lemma4}). The proof of Theorem \ref{mainthm} will then be done in
Section \ref{section5}.

\section{Positive lower bound for the  speed}\label{section2}

Here is our first key lemma.

\begin{lemma}[Lower bound for $c$] \label{lemma1}Assume \eqref{HypsurRK} and let $c>0,$ $\psi \in
W^{2,\infty }(\mathbb{T})$ satisfy \eqref{eqn_front_laplacien} {\it a.e.} and $%
u\in H^{1}(\Omega )$ a variationnal solution of \eqref{Eqn_laplacien}.
Then 
\begin{equation*}
c\geq R_{m}\int\limits_{0}^{1}K(s) \, ds.
\end{equation*}
\end{lemma}

To prove Lemma \ref{lemma1} we need some technical results.

\begin{lemma}\label{lemma2}
Let $(c,\psi ,u)$ be given as in Lemma \ref{lemma1}. Then $\int\limits_{\mathbb{T}%
}u(\psi (y),y) \, dy=1.$
\end{lemma}

\begin{proof}
 Consider any arbitrary $L>0$ and take
\begin{equation*}
w(x,y)
:=
\begin{cases}
0 & \text{if }x\leq -L, \\ 
x+L  & \text{if } -L\leq x\leq -L+1, \\ 
1 & \text{if } -L+1\leq x\leq \psi (y)%
\end{cases}
\end{equation*}%
 in \eqref{weaksolution}, which is in $H^{1}(\Omega )$ since $\psi$ is bounded. This gives
\begin{equation}
\int\limits_{\mathbb{T}} \Big(\int_{-L}^{-L+1}c(x+L)u_{x} \, dx+\int_{-L+1}^{\psi (y)}cu_{x} \, dx+\int_{-L}^{-L+1}u_{x} \, dx\Big) dy=c.
\label{Eqnw_L}
\end{equation}
\noindent After an integration by parts, the sum of the first two terms gives%
\begin{equation*}
c\int\limits_{\mathbb{T}}u(\psi (y),y) \, dy-c\int\limits_{\mathbb{T}%
}\int_{-L}^{-L+1}u \, dxdy
\end{equation*}
\noindent where the second integral vanishes when  $L\rightarrow +\infty $ 
since 
\begin{equation*}
0\leq c\int\limits_{\mathbb{T}}\int_{-L}^{-L+1}u \, dxdy\leq
c \, \|u\|
_{L^{2}((-L,-L+1)\times \mathbb{T)}}.
\end{equation*}
Likewise the third term of \eqref{Eqnw_L} goes to zero when  $L\rightarrow +\infty $  and we conclude the proof by passing to the limit in 
\eqref{Eqnw_L}.
\end{proof}

We proceed by a second technical lemma.

\begin{lemma}\label{lemma3}
Let $(c,\psi ,u)$ as in  Lemma \ref{lemma1}.  Then $\int\limits_{\mathbb{T}}K(u(\psi
(y),y) \, dy\geq \int_{0}^{1}K(s) \, ds.$
\end{lemma}

\begin{proof}
For any $\varepsilon >0,$ define $K_{\varepsilon }(u)
:=
\frac{1}{\varepsilon }%
\int_{(u-\varepsilon )^{+}}^{u}K(s) \, ds.$ Then $K_{\varepsilon }$ is Lipschitz
and satisfies \eqref{HypsurRK}. In particular we can take $%
w(x,y) 
:=
K_{\varepsilon }(u(x,y))$ in \eqref{weaksolution}. We thus have 
\begin{equation*}
c\int_{\Omega }u_{x}K_{\varepsilon }(u)+\int_{\Omega }\nabla u\nabla
K_{\varepsilon }(u)=\int\limits_{x=\psi (y)}\frac{c \, K_{\varepsilon
}(u)}{\sqrt{1+\psi _{y}^{2}}},
\end{equation*}

\noindent which can be rewritten as 
\begin{equation*}
c\int_{\Omega }J_{\varepsilon }(u)_{x} \, dxdy+\int_{\Omega }K_{\varepsilon
}^{\prime }(u)| \nabla u| ^{2} dxdy=\int\limits_{\mathbb{T}%
}c\, K_{\varepsilon }(u(\psi (y),y)) \, dy
\end{equation*}

\noindent where $J_{\varepsilon }(u)
:=
\int_{0}^{u}K_{\varepsilon }(s) \, ds.$ As
the second term is nonnegative by \eqref{HypsurRK} and since 
\begin{equation*}
c\int_{\Omega }J_{\varepsilon }(u)_{x} \, dxdy=c\int\limits_{\mathbb{T}%
} \Big(\int_{-\infty }^{\psi (y)}J_{\varepsilon }(u)_{x} \, dx\Big) dy=c\int\limits_{%
\mathbb{T}}J_{\varepsilon }(u(\psi (y),y)) \, dy,
\end{equation*}

\noindent we deduce that 
\begin{equation*}
\int\limits_{\mathbb{T}}K_{\varepsilon }(u(\psi (y),y)) \, dy \geq \int\limits_{%
\mathbb{T}}J_{\varepsilon }(u(\psi (y),y)) \, dy\geq J_{\varepsilon
}\Big(\int\limits_{\mathbb{T}}u(\psi (y),y) \, dy\Big)=\int_{0}^{1}K_{\varepsilon }(s) \, ds
\end{equation*}

\noindent thanks to Jensen's inequality and Lemma \ref{lemma2}. We then complete the
proof by using the fact that $K_{\varepsilon }\uparrow K$ as $\varepsilon
\downarrow 0.$
\end{proof}

 We are now ready to prove Lemma \ref{lemma1}.

\begin{proof}[Proof of Lemma \ref{lemma1}]
An integration of \eqref{eqn_front_laplacien} over $\mathbb{T}$ gives%
\begin{equation*}
c=\int\limits_{\mathbb{T}}R(y)K(u(\psi (y),y))\sqrt{1+\psi _{y}^{2}} \, dy.
\end{equation*}%
We thus have $c\geq R_{m}\int\limits_{\mathbb{T}}K(u(\psi (y),y)) \, dy$ and we
conclude by Lemma \ref{lemma3}.
\end{proof}

\section{A monotonicity result}\label{section3}

Here is our second key lemma.

\begin{lemma}[Monotonicity of u]\label{lemma4} Let $c>0,$ $\psi \in W^{2,\infty }(\mathbb{T})$ and $%
u\in H^{1}(\Omega )$ satisfy \eqref{Eqn_laplacien}. Then the function 
\begin{equation*}
x\in (-\infty ,\max \psi ]\longmapsto \min_{y:(x,y)\in \overline{\Omega }%
}u(x,y)\text{ is nondecreasing}.
\end{equation*}
\end{lemma}

\begin{remark}
Recall that $u$ is defined on $\overline{\Omega }=\{x\leq \psi (y)\}$, which
explains the above interval of definition.
\end{remark}

\begin{remark}\label{remark2}
The min above can be understood in the classical sense since a variationnal
solution $u\in H^{1}(\Omega )$ of \eqref{Eqn_laplacien} is at least
 in $C(\overline{\Omega})$  thanks to classical elliptic regularity
results (see \textit{e.g.} \cite{Nit11} and the references therein or  \cite[Appendix A.4]{IFB}). 
\end{remark}

\begin{proof}
Let us take $x_{0}\in (-\infty ,\max \psi )$ and consider the problem 
\begin{equation}
\begin{cases}
c\,w_{x}-\triangle w=0, & \quad \text{in }\Omega _{0}, \\[1ex] 
\frac{\partial w}{\partial \nu }=\frac{c\,}{\sqrt{1+\psi _{y}^{2}}}, & 
\quad \text{on }\partial \Omega _{0}\cap \{x>x_{0}\}, \\[1ex] 
w=w_{0}, & \quad \text{on }\partial \Omega _{0}\cap \{x=x_{0}\},%
\end{cases}
\label{system_en_w}
\end{equation}

\noindent where 
\begin{equation*}
\Omega _{0}
:=
\Omega \cap \{x>x_{0}\}\text{ and }w_{0}
:=
\min_{y}u(x_{0},y).
\end{equation*}

\noindent Then the constant function $w_{0}$ is a subsolution of \eqref%
{system_en_w} whereas $u$ is a supersolution. By comparison, 
\begin{equation}\label{missing}
u(x,y)\geq w_{0}\text{ for all }(x,y)\in \overline{\Omega }_{0}
\end{equation}

\noindent and it follows that 
\begin{equation*}
\min_{y:(x,y)\in \overline{\Omega}_0}u(x,y)\geq w_{0}=\min_{y:(x,y)\in 
\overline{\Omega}_0}u(x_{0},y)\text{ for all }x\in \lbrack x_{0},\max
\psi ].
\end{equation*}
\end{proof}

\begin{remark}\label{remark3}
Note that $\partial \Omega _{0}$ may not be Lipschitz at truncated points. Nevertheless the comparison result used in the above proof holds by standard arguments \cite{Bre83,Grisvard}.  For completeness, a  short  verification is given in Appendix \ref{appmissing}.
\end{remark}

\section{Positive lower bound for the temperature}\label{section4}

In this section we prove a general result concerning the temperature.
Consider for this sake the following problem 
\begin{equation}
\begin{cases}
c\,u_{x}-\triangle u=0, & \quad \text{in }\Omega , \\[1ex] 
\frac{\partial u}{\partial \nu }=\frac{c\,}{\sqrt{1+\psi _{y}^{2}}}, & 
\quad x=\psi (y), \\[1ex] 
u(x,y)\rightarrow 0, & \quad \mbox{as~$x\rightarrow -\infty$,}%
\end{cases}
\label{Eqnlaplacienbis}
\end{equation}%
where $c\in \mathbb{R}$ and $\psi \in W^{2,\infty }(\mathbb{T})$ satisfy,
for some given $c_{0}$, $c_{M}$ and a continuous function $G$,%
\begin{equation}
\begin{cases}
0<c_{0}\leq c\leq c_{M}, \\[1ex]
| \psi _{yy}| \leq G(\psi _{y})\text{ {\it a.e.} and }%
\min\limits_{
 y \in \mathbb{T} 
}\psi (y)=0.%
\end{cases}
\label{hyp_sur_cv}
\end{equation}%

We then claim our last key result.

\begin{lemma}[Lower bound on $u$]\label{lemma5} Let  $G\in C(\mathbb{R})$  and $0<c_{0}\leq c_{M}$
be given. Then there exists $\alpha >0$ such that for all $c\in \mathbb{R}$
and $\psi \in W^{2,\infty }(\mathbb{T})$ verifying \eqref{hyp_sur_cv}, the
unique variationnal solution $u\in H^{1}(\Omega )$ of \eqref{Eqnlaplacienbis}
satisfies
\begin{equation*}
\min_{
 y \in \mathbb{T} 
}u(0,y)\geq \alpha .
\end{equation*}
\end{lemma}

\begin{proof} 
Let us define%
\begin{equation*}
\alpha 
:=
\inf_{c,\psi } \left\{\min_{y}u(0,y); \, u\in H^{1}(\Omega )\text{
solution of \eqref{Eqnlaplacienbis}}\right\},
\end{equation*}%
the infimum being taken on all $c\in \mathbb{R}$ and $\psi \in W^{2,\infty }(%
\mathbb{T})$ satisfying \eqref{hyp_sur_cv}. The $\min $ is to be
understood in the classical sense by Remark \ref{remark2}. We have to show that $\alpha $
is positive. Consider for this sake a minimizing triplet $(c_{n},\psi
_{n},u_{n})$ such that 
\begin{equation}
\left\{ 
\begin{array}{c}
c_{n}\text{ and }\psi _{n}\in W^{2,\infty }(\mathbb{T})\text{ satisfy \eqref%
{hyp_sur_cv}}, \\ 
\\ 
u_{n}\in H^{1}(\Omega _{n}
 := 
\{x<\psi _{n}(y)\})\text{ solves \eqref%
{Eqnlaplacienbis}}\text{ } \\ 
\text{and }\lim_{n\rightarrow \infty }\text{ }\min_{y}u_{n}(0,y)=\alpha .%
\end{array}%
\right.  \label{Hyptriplet}
\end{equation}%
Since $u_{n} \in C(\overline{\Omega }_{n})$ by Remark \ref{remark2}  and  $\mathbb{T}$ and $[c_{0},c_{M}]$ are compact, there exist $(y_{n})_{n}$  and $\overline{y}$ in $\mathbb{T}$ as well as  $\overline{c}\in [ c_{0},c_{M}]$
such that%
\begin{equation}\label{CVcnyn}
\begin{array}{c}
\min_{y}u_{n}(0,y)=u_{n}(0,y_{n}),   \\
y_{n}\rightarrow \overline{y}\text{ and }c_{n}\rightarrow 
\overline{c}\text{ when }n\rightarrow \infty.
\end{array}%
\end{equation}
Consider now the limiting domain 
\begin{equation*}
\Omega _{\infty }
:=
\cup _{n}Int(\cap _{k\geq n}\Omega _{k}),
\end{equation*}%
 where $Int$ is the interior.  Note that $\Omega _{\infty }$ is open. Note also that $(0,\overline{y})$ $%
\in $ $\overline{\Omega }_{\infty }$ (the closure of $\Omega _{\infty })$
since each $\psi _{n}$ was chosen such that $\psi _{n}\geq 0,$ so that for
all $\varepsilon >0,$ $(-\varepsilon ,\overline{y})\in \Omega _{\infty }.$
The analysis which follows will depend on whether $(0,\overline{y})$ is on $%
\partial \Omega _{\infty }$ (the boundary of $\Omega _{\infty })$ or in $%
\Omega _{\infty }$. In the sequel, $B(y,r):=\{z\in \mathbb{T}; \, d(y,z)<r\}$ will denote a ball of $\mathbb{T}$ with $d(y,z):=dist(P^{-1}(y),P^{-1}(z))$ and where $P: \mathbb{R\rightarrow
R}/\mathbb{Z}$ is the usual projection. We denote similarly by $B((x,y),r)$
the balls of $\mathbb{R}\times \mathbb{T}$.

\bigskip

\textbf{First case : }$(0,\overline{y})\in \partial \Omega
_{\infty }.$

\begin{claim}\label{claim1}
In this case we claim that there exist an increasing sequence $(n_p)_p$ in $\mathbb{N}$ and $(w_p)_p$ in $\mathbb{T}$ such that 
\begin{equation}
w_{p}\rightarrow \overline{y}\text{ and }\psi
_{n_{p}}(w_{p})\rightarrow 0\text{ as }p\rightarrow \infty .
\label{CVwp}
\end{equation}
\end{claim}

\noindent Indeed $B((0,\overline{y}),\frac{1}{p})$ intersects $\Omega
_{\infty }^{c}=\cap _{n}(\overline{\cup _{k\geq n}\Omega _{k}^{c}})$ for any 
 integer $p \geq 1$,  so that there exists a sequence $(x_{p},y_{p})$ such that 
$(x_{p},y_{p})\in B((0,\overline{y}),\frac{1}{p})\cap (\overline{\cup
_{k\geq p}\Omega _{k}^{c}})$. Since $B((0,\overline{y}),\frac{1}{p})$ is
open and $(x_{p},y_{p})$ is the limit of some sequence in $\cup _{k\geq
p}\Omega _{k}^{c},$ there exist $n_{p}\geq p$ and $(\tilde{x}_{p},%
\tilde{y}_{p})$ such that $(\tilde{x}_{p},\tilde{y}_{p})\in B((0,%
\overline{y}),\frac{1}{p})\cap \Omega _{n_{p}}^{c}.$ Taking $w_{p}=%
\tilde{y}_{p}$ gives \eqref{CVwp}. Indeed  $w_{p} \in B(\overline{y%
},\frac{1}{p})$  and $0\leq \psi _{n_{p}}(w_{p})\leq 
\tilde{x}_{p}\leq \frac{1}{p}$ by construction. Moreover the sequence $%
n_{p}$ can be chosen increasing up to taking subsequences if necessary since 
$n_{p}\geq p$ goes to infinity as $p\rightarrow \infty .$  This proves Claim \ref{claim1}. 

\medskip

We proceed by
another claim:

\begin{claim}\label{claim2}
For all $\varepsilon >0$ and $p_{0}\in \mathbb{N},$ there exist $y\in B(%
\overline{y},\varepsilon )$ and $p\geq p_{0}$ such that 
\begin{equation}
| \psi _{n_{p}}^{\prime }(y)| <\varepsilon .
\label{Cvwpprime}
\end{equation}
\end{claim}

\noindent Indeed if the above does not hold, then there exist $\varepsilon
>0 $ and $p_{0}\in \mathbb{N}$ such that for any $p\geq p_{0},$ 
\begin{equation*}
|\psi _{n_{p}}^{\prime }(y)| \geq \varepsilon \text{ for all }y\in B(\overline{y},\varepsilon ).
\end{equation*}

\noindent Up to taking a subsequence, we can suppose that for all $p,$ $\psi
_{n_{p}}^{\prime }(y)\geq \varepsilon $ in $B(\overline{y},\varepsilon ).$
For $\varepsilon $ small enough, identifying $B(\overline{y},\varepsilon )$
with the interval $I
:=
P^{-1}(B(\overline{y},\varepsilon ))$ in $\mathbb{R}$,
we have in  $I$  
\begin{eqnarray*}
\psi _{n_{p}}(\overline{y}-\varepsilon ) &=&\psi
_{n_{p}}(w_{p})+\int_{w_{p}}^{\overline{y}-\varepsilon }\psi
_{n_{p}}^{\prime }(y) \, dy \\
&\leq &\psi _{n_{p}}(w_{p})-\varepsilon (w_{p}-(\overline{y}-\varepsilon )).
\end{eqnarray*}

\noindent This is not possible because as $\psi _{n_{p}}\geq 0$ by
definition, the LHS is $\geq 0$ whereas the RHS goes to $-\varepsilon ^{2}<0$
in the limit $p$ $\rightarrow \infty $.  This proves Claim \ref{claim2}. 

\medskip

Now \eqref{Cvwpprime} implies that there exist a subsequence ($%
n_{p_{k}})_{k}$ and another sequence $(\tilde{w}_{k})_{k}$ such that 
\begin{equation*}
\tilde{w}_{_{k}}\rightarrow \overline{y}\text{ and }\psi
_{n_{p_{k}}}^{\prime }(\tilde{w}_{k})\rightarrow 0\text{ as }k\text{ 
}\rightarrow \infty \text{.}
\end{equation*}%
Recalling \eqref{CVwp} and renaming the sequences from the beginning of
the proof for simplicity, we have proved that we have $(w_{n})_{n}$ and $(%
\tilde{w}_{n})_{n}$ such that 
\begin{equation}
w_{n}\rightarrow \overline{y}\text{, }\tilde{w}_{n}\rightarrow 
\overline{y}\text{ with }\psi _{n}(w_{n})\rightarrow 0\text{ and }\psi
_{n}^{\prime }(\tilde{w}_{n})\rightarrow 0\text{ as }n
\rightarrow \infty \text{.}  \label{Cvwn}
\end{equation}
This leads to  our next claim below concerning the fronts. 

\begin{claim}\label{claim3}
There exists $\varepsilon _{0}$ such that $(\psi _{n})_{n}$
is bounded in $W^{2,\infty }(B(\overline{y},\varepsilon _{0})).$ In
particular $\psi _{n}\rightarrow \overline{\psi }$ uniformly in $B(%
\overline{y},\varepsilon _{0}),$ up to some subsequence, for some $\overline{%
\psi } \in W^{2,\infty }(B(\overline{y},\varepsilon _{0}))$.
\end{claim}

\noindent  Indeed  the existence of $\varepsilon _{0}$ as well as the uniform bound for 
$\psi _{n}^{\prime }$ in $B(\overline{y},\varepsilon _{0})$ are provided by
Lemma \ref{lemma7}  in  Appendix \ref{appmissing} and \eqref{Cvwn}. The bounds for $\psi _{n}$
follow from \eqref{Cvwn} again and the ones for $\psi _{n}^{\prime \prime
} $ from \eqref{hyp_sur_cv}. With these bounds in hand, Ascoli-Arzela
theorem completes the proof of Claim \ref{claim3}.

\medskip

 To conclude, we will pass to the limit in the boundary problem satisfied by the temperature, cf. \eqref{Eqnlaplacienbis}. In the case where $(0,\overline{y}) \in \partial \Omega_\infty$, we only need to handle the boundary condition on the moving interface $\{x=\psi_n(y), \, y \in B(\overline{y},\varepsilon_0)\}$. Before hand, it is convenient to extend  $u_{n}$ onto  $\mathbb{R} \times \mathbb{T}$ by  

\begin{equation}
\tilde{u}_{n}(x,y)
:=
\begin{cases}
u_{n}(x,y) & \text{if } (x,y)\in \Omega _{n}, \\ 
u_{n}(2\psi _{n}(y)-x,y) & \text{otherwise.}
\end{cases}
\label{Extensionun}
\end{equation}

\noindent  Consider now the cylinder $\tilde{Q}:= \{(x,y); \, y \in B(\overline{y},\varepsilon_0)\}$. We then have 

\begin{claim}\label{claim4}
There exists  $0 \leq \tilde{u} \in C^{1}(\tilde{Q})$  such that%
\begin{equation*}
\tilde{u}_{n}\rightarrow \tilde{u}\text{ locally uniformly in }%
\tilde{Q}\text{ as }n\rightarrow \infty ,
\end{equation*}%
up to a subsequence,  with
\begin{equation*}
\frac{\partial \tilde{u}}{\partial \nu }=\frac{\overline{c\,}}{\sqrt{1+%
\overline{\psi }_{y}^{2}}} \text{ on } \left\{x=\overline{\psi }(y), \, y\in B(%
\overline{y},\varepsilon _{0})\right\}.
\end{equation*}
\end{claim}

\noindent To show this claim, apply Lemma \ref{lemmaA1} of Appendix \ref{appA} with $
 B_0 
 :=  
B(%
\overline{y},\varepsilon _{0})$ and  $
 B_1 
 := 
\emptyset$.
Note that \eqref{A2} holds by  \eqref{CVcnyn} and  Claim \ref{claim3}.

\medskip

To conclude we have $u_{n}(0,y_{n})\rightarrow \tilde{u}(0,%
\overline{y})=\alpha $  as $n \to \infty$, cf. \eqref{Hyptriplet}, \eqref{CVcnyn}, \eqref{Cvwn}, as well as Claims \ref{claim3} and \ref{claim4}, and it remains now  to show
that $\tilde{u}(0,\overline{y})>0.$ But as $\tilde{u}\geq 0$  around $(0,\overline{y})$  and $%
\frac{\partial \tilde{u}}{\partial \nu }(0,\overline{y})>0$, 
we end up with $\tilde{u}(0,\overline{y})>0.$ This completes the proof  of Lemma \ref{lemma5} when 
$(0,\overline{y})\in \partial \Omega _{\infty }.$

\bigskip

\textbf{Second case : }$(0,\overline{y})\in \Omega _{\infty }.$

\smallskip

\noindent Recalling that $\Omega _{\infty }=\cup _{n}Int(\cap _{k\geq
n}\Omega _{k}),$ there exist now $n_{0}$ and an open $O\subset \mathbb{R}%
^{2} $ such that 
\begin{equation}
(0,\overline{y})\in O\subseteq \Omega _{n}\text{ for all }n\geq n_{0}.
\label{4etoiles}
\end{equation}

\noindent Roughly speaking, we will show that $u_{n}$ converges to some  
nontrivial $\tilde{u} \geq 0$  which satisfies the first equation of \eqref%
{Eqnlaplacienbis} in $O$.  This will imply that $\alpha =\tilde{u}(0,\overline{y})>0$ by an argument of propagation of maximum. The overall idea to get a nontrivial limit is  to work eventually in a
larger open $\mathcal{O}$ with $O\subseteq \mathcal{O}\subseteq \Omega
_{\infty },$ so that the  nontrivial  boundary condition of \eqref%
{Eqnlaplacienbis} holds on some part of $\partial \mathcal{O}$. 

\medskip

To construct $\mathcal{O}$, take $(z_{n})_{n}$ in $\mathbb{T}$
such that $z_{n}$ is a minimizer of $\psi _{n}.$ Considering a
subsequence if necessary, we can assume that 
\begin{equation}
\psi _{n}(z_{n})=\psi _{n}^{\prime }(z_{n})=0\text{ and }%
z_{n}\rightarrow \overline{z}\text{ as }n\rightarrow \infty ,  \label{Cvzn}
\end{equation}

\noindent for some $\overline{z}\in \mathbb{T}$. The above property looks
like \eqref{Cvwn} and we have the following analogous of Claim \ref{claim3}  whose proof is similar. 

\begin{claim}\label{claim5}
There exist $\varepsilon _{0}$ and $\overline{\psi } \in
W^{2,\infty }(B(\overline{z},\varepsilon _{0}))$ such that  $(\psi _{n})_{n}$
is bounded in $W^{2,\infty }(B(\overline{z},\varepsilon _{0}))$ and  $\psi _{n}\rightarrow \overline{\psi }$ uniformly in $B(%
\overline{z},\varepsilon _{0})$,  up to some
subsequence. 
\end{claim}

Using now \eqref{4etoiles}, there exist $\eta >0$ and $\varepsilon _{1}>0$
such that  $(-\eta ,\eta) \times B(\overline{y},\varepsilon _{1})\subset
\Omega _{n}$  for all $n\geq n_{0}.$ 
Recall that $\Omega _{n}=\{x<\psi _{n}(y)\}$ so that we have 
\begin{equation*}
\psi _{n}(y)\geq \eta, \text{ for all $n\geq n_{0}$ and $y\in B(\overline{y},\varepsilon _{1})$}.
\end{equation*}%
Using in addition \eqref{Cvzn} and Claim \ref{claim5}, we note in particular that $%
\overline{y}\neq \overline{z}.$ This enables us to choose $\varepsilon _{0}$
and $\varepsilon _{1}$ (smaller if needed) so that $B(\overline{z}%
,\varepsilon _{0})\cap B(\overline{y},\varepsilon _{1})=\emptyset .$
We can now define $\mathcal{O}$ (resp. $\tilde{\mathcal{O}})$
as follows:

\begin{equation*}
\mathcal{O} \text{ (resp. $\tilde{\mathcal{O}}$)}
:=
\left\{(x,y); \, 
x<\chi (y)\text{ (resp. $x<\tilde{\chi }(y)$)}\right\},
\end{equation*}%
where%
\begin{equation*}
\chi (y) \text{ (resp. $\tilde{\chi }(y)$)}
:=
\begin{cases}
\overline{\psi }(y) \text{ (resp. $+\infty$)} & \text{if }y\in B(\overline{z}%
,\varepsilon _{0}), \\ 
\eta & \text{if }y\in B(\overline{y},\varepsilon _{1}), \\ 
0 & \text{otherwise.}%
\end{cases}
\end{equation*}
 As previously announced, we need to identify both the PDE and the boundary condition of the limiting temperature. We propose

\begin{claim}
There exists  $0 \leq \tilde{u} \in C^1(\tilde{\mathcal{O}}) \cap C^2(\mathcal{O})$  such
that%
\begin{equation*}
\tilde{u}_{n}\rightarrow \tilde{u}\text{ locally uniformly in }%
\tilde{\mathcal{O}}\text{ as }n\rightarrow \infty ,
\end{equation*}%
up to a subsequence,  where 
\begin{equation}
\begin{cases}
\overline{c}\,\tilde{u}_{x}-\triangle \tilde{u}=0,\text{ in }%
\mathcal{O}, \\ 
\frac{\partial \tilde{u}}{\partial \nu }=\frac{\overline{c\,}}{\sqrt{1+%
\overline{\psi }_{y}^{2}}},\text{ }x=\overline{\psi }(y),\text{ }y\in B(%
\overline{z},\varepsilon _{0}). 
\end{cases}\label{Lemme8prime}
\end{equation}
\end{claim}
\noindent To show this claim, apply again Lemma \ref{lemmaA1} with now $
 B_0 
 := 
B(\overline{z},\varepsilon _{0}),$ $\eta $
as above and $
 B_1 
 := 
B(\overline{y},
 \varepsilon _{1} 
),$ so \eqref{A2}  is  
ensured by Claim \ref{claim5} and \eqref{A3}  holds  by what precedes. 

\medskip

 Once again 
$u_{n}(0,y_{n})\rightarrow \tilde{u}(0,\overline{y})=\alpha$
 and it remains to  show that $\tilde{u}(0,\overline{y})>0.$  For this sake, note that $\tilde{u}$ is not 
identically zero because $\overline{c}>0$ in the
boundary condition of \eqref{Lemme8prime}. By the strong  maximum principle, 
$\tilde{u}(0,\overline{y})>0$ since it cannot
achieve its minimum in  the connected open  $\mathcal{O}$ otherwise it will vanish
everywhere in $\mathcal{O}$, cf.  \cite[Sec. 6.4.2]{Evans}.  This completes
the proof of Lemma \ref{lemma5}.
\end{proof}

\section{Proof of Theorem \ref{mainthm}}\label{section5}

Now we are ready to prove our main result. 
 For this sake, let us consider Eqns. \eqref{Eqn_laplacien}-\eqref{eqn_front_laplacien}
with $K(u)$ replaced by the truncated function 
\begin{equation*}
K_{n}(s)
:=
\max \left\{K(s),\frac{1}{n}\right\}.
\end{equation*}%
By \cite[Thm. 3.2]{IFB}, the latter system
admits a solution $(c_{n},\psi _{n},u_{n})$ for any  integer $n \geq 1$,  with $%
c_{n}>0,$ $\psi _{n}\in W^{2,\infty }(\mathbb{T})$ and  $0 \leq u_{n}\in
H^{1}(\{x<\psi_n(y)\})$. Since $\psi _{n}$ is defined up
to an additive constant,  we can choose it such that  $\min_{y}\psi _{n}(y)=0.$ Now from Lemma \ref{lemma1}, we know that 
\begin{equation*}
c_{n}\geq c_{0} 
:= 
R_{m}\int_{0}^{1}K(s) \, ds>0\text{ for all }n.
\end{equation*}%
Likewise we have the following upper bound by  
\cite[Thm. 2.1]{ChNa97}  and Assumption \eqref{HypsurRK}:  
\begin{equation*}
c_{n}\leq c_{M}
:=R_M
 K_M 
\text{ for all }n.
\end{equation*}%
From the front's equation \eqref{eqn_front_laplacien}, we then have 
\begin{equation*}
| (\psi _{n})_{yy}| \leq 2 c_{M} (1+\psi _{ny}^{2})^{%
\frac{3}{2}}.
\end{equation*}%
 Hence, applying Lemma \ref{lemma5} with $G(h)
:=
2c_{M}(1+h^{2})^{\frac{3}{2}}$   leads to the existence of an $\alpha >0,$ $\alpha $ depending only
on $c_{0},$ $c_{M}$ and $G,$ such that 
\begin{equation*}
\min_{
 y 
}u_{n}(0,y)\geq \alpha \text{ for all }n.
\end{equation*}%
Now as $x\longmapsto \min_{y}u(x,y)$ is  nondecreasing  by Lemma \ref{lemma4}, we  have
$u_{n}(\psi _{n}(y),y)\geq \alpha,$ 
and consequently by \eqref{HypsurRK} and the definition of $K_{n},$ we
 end up with 
\begin{equation*}
K_{n}(u_{n}(\psi _{n}(y),y))\geq K_{n}(\alpha )\geq K(\alpha )\text{
for all }n
 \text{ and } y \in \mathbb{T}. 
\end{equation*}%
We proceed by setting 
\begin{equation*}
H_{n}(y)
:=
R(y) K_{n}(u_{n}(\psi _{n}(y),y)).
\end{equation*}%
 Since $%
K(\alpha )>0$ by \eqref{HypsurRK}, the above uniform positive lower boundedness of $(H_{n})_{n}$  
enables us to use the results of \cite%
{IFB} to pass to the limit  as  $n \rightarrow \infty $. Indeed an application
of  \cite[Lemmas  
2.5 \& 2.7]{IFB}  gives the existence   of $(H,c,\psi,u)$ limit of $(H_{n},c_{n},\psi
_{n},u_{n})$, 
 up to a subsequence,  where  
\begin{equation*}
\begin{cases}
c\,u_{x}-\triangle u=0, & \quad \text{in }\{x<\psi (y)\}, \\[1ex] 
\frac{\partial u}{\partial \nu }=\frac{c\,}{\sqrt{1+\psi _{y}^{2}}}, & 
\quad x=\psi (y), \\[1ex] 
u(x,y)\rightarrow 0, & \quad \mbox{as~$x\rightarrow -\infty$,}%
\end{cases}%
\end{equation*}%
and 
\begin{equation*}
-c+H(y)\sqrt{1+\psi _{y}^{2}}=\,\frac{\psi _{yy}}{1+\psi _{y}^{2}},\quad
x=\psi (y),
\end{equation*}%
 in the sense and with the regularity claimed in Theorem \ref{mainthm}.  It remains to identify $H(y).$ As  $H_{n}\rightharpoonup H$  in $L^{\infty }$ $%
weak-\ast $ (see \cite{IFB})$,$ we have%
\begin{equation*}
\int_{\mathbb{T}}H_{n}(y)\varphi (y) \, dy\rightarrow \int_{\mathbb{T}%
}H(y)\varphi (y) \, dy\text{ for all }\varphi \in C_{c}^{\infty }(\mathbb{T)}%
\text{.}
\end{equation*}%
 But by \cite[Lemma 2.8]{IFB},  we have 
\begin{equation*}
u_{n}(\psi _{n}(y),y)\rightarrow u(\psi (y),y)\text{ for  {\it a.e.} }%
y\in \mathbb{T}\text{,}
\end{equation*}%
so that, by the  definition of $H_{n}$, the assumptions in \eqref{HypsurRK},  and Lebesgue's theorem, we also have%
\begin{equation*}
\int_{\mathbb{T}}H_{n}(y)\varphi (y) \, dy\rightarrow \int_{\mathbb{T}%
}R(y)K(u(\psi (y),y))\varphi (y) \, dy\text{ for all }\varphi \in
C_{c}^{\infty }(\mathbb{T)}\text{.}
\end{equation*}%
Finally we end up with $H(y)=R(y) K(u(\psi (y),y))$  and the proof of Theorem \ref{mainthm} is complete. \qed  

\appendix

\section{A convergence result for the temperature}\label{appA}

Let us now state and prove Lemma \ref{lemmaA1}, see below, that we have admitted during the proof of Lemma \ref{lemma5}. As before, consider $(c_{n})_n$ in $\R,$ $(\psi _{n})_n$ in $W^{2,\infty }(%
\mathbb{T})$ and $(u_{n})_n$ in $H^1(\Omega_n=\{x<\psi
_{n}(y)\})$ such that 
\begin{equation}\label{A1}
\begin{cases}
c_{n}\,(u_{n})_{x}-\triangle u_{n}=0, & \quad \text{in }\Omega _{n}\\[1ex] 
\frac{\partial u_{n}}{\partial \nu }=\frac{c_{n}}{\sqrt{1+\psi _{ny}^{2}}},
& \quad x=\psi _{n}(y), \\[1ex] 
u_{n}(x,y)\rightarrow 0, & \quad \mbox{as~$x\rightarrow -\infty$},%
\end{cases}
\end{equation}
as well as
\begin{equation}\label{A2}
\begin{cases}
0<c_n \leq c_M \mbox{ and } c_n \to \overline{c}>0 \mbox{ as } n \to  \infty,\\
(\psi _{n})_{n}\text{ is bounded in }W^{2,\infty }(B_0), \text{ and} 
 \\
\psi _{n}\rightarrow \overline{\psi }\text{ uniformly in }B_0\text{ as }%
n\rightarrow \infty ,  
\end{cases}
\end{equation}
for some reals $c_M$ and $\overline{c}$, nonempty open $B_0\subset \mathbb{T}$ and $\overline{\psi }\in
W^{2,\infty }(B_0).$ 
Let moreover $\eta > 0$ and $B_1\subset \mathbb{T}$
be another open, eventually empty, such that 
\begin{equation}\label{A3}
B_0\cap B_1 = \emptyset \text{ and } (-\eta,\eta) \times B_1\subset \cap _{n}\Omega _{n}.  
\end{equation}

\noindent  Define then 
\begin{equation*}
U_{n}
:=
\{(x,y); \, x<\chi _{n}(y)\},
\end{equation*}

\noindent where 
\begin{equation*}
\chi _{n}(y) 
:=
\begin{cases}
\psi _{n}(y) & \text{if }y\in B_0, \\ 
\eta & \text{if }y\in B_1, \\ 
0 & \text{otherwise.}
\end{cases}
\end{equation*}

\noindent  Likewise we define the limiting domain 
\begin{equation*}
U
:=
\{(x,y); \, x<\chi (y)\},
\end{equation*}
\noindent  with $\chi $ the same as $\chi _{n}$ with $\psi _{n}$ replaced by $\overline{%
\psi }.$
We further consider the extended domain 
\begin{equation*}
\tilde{U}
:=
\left\{(x,y); \, x<\tilde{\chi }(y)\right\},
\end{equation*}
\noindent where 
\begin{equation*}
\tilde{\chi }(y)
:=
\begin{cases}
+\infty  & \text{if }y\in B_0, \\ 
\eta & \text{if }y\in B_1, \\ 
0 & \text{otherwise.}%
\end{cases}
\end{equation*}
Recall that $\tilde{u}_{n}$ is defined in \eqref{Extensionun} and  belongs at least to $C\cap H^{1}(
\mathbb{R} \times \mathbb{T}
)$ because   $u_{n}\in C\cap H^{1}(\Omega _{n})$ and $\psi _{n}\in
W^{1,\infty }(\mathbb{T}).$ 

\begin{lemma}\label{lemmaA1}
Under the above assumptions, there is $0 \leq \tilde{u}$ $\in
C^{1}(\tilde{U}) \cap C^2(U)$ such that%
\begin{equation*}
\tilde{u}_{n}\rightarrow \tilde{u}\text{ locally uniformly in }%
\tilde{U}\text{ as }n\rightarrow \infty ,
\end{equation*}%
up to a subsequence, where 
\begin{equation}\label{A5}
\begin{cases}
\overline{c}\,\tilde{u}_{x}-\triangle \tilde{u}=0,\text{ in }U, \\ 
\frac{\partial \tilde{u}}{\partial \nu }=\frac{\overline{c\,}}{\sqrt{1+%
\overline{\psi }_{y}^{2}}},\text{ }x=\overline{\psi }(y),\text{ }y\in B_0. 
\end{cases}
\end{equation}
\end{lemma}

\begin{proof} We proceed in several steps.

\bigskip

{\bf Step 1:} {\it A priori estimates.}

\smallskip

\noindent  We claim that 
\begin{equation}
\begin{cases}
0\leq u_{n}(x,y)\leq e^{c_n x} \quad \forall (x,y) \in \Omega_n\\
\text{and } \int_{\Omega
_{n}}| \nabla u_{n}|^{2}\leq c_{n}.
\end{cases}
\label{EstimationsurU}
\end{equation}

\noindent The first estimate follows since the functions $0$ and $e^{c_{n}x}$ are respectively sub and supersolution of \eqref{A1}, see \textit{e.g.} \cite[Appendix A.2]{IFB}.
For the second estimate, take $w:=u_{n}$ in \eqref%
{weaksolution}, which is a weak formulation of \eqref{A1} as well, to see that
\begin{equation*}
c_{n}\int_{\Omega } \left(u_{n}^{2}/2\right)_{x}+\int_{\Omega }| \nabla
u_{n}|^{2}=\int\limits_{x=\psi_(y)}u_{n}\frac{c_{n}}{\sqrt{%
1+(\psi _{n}^{\prime })^{2}}} \, dy=c_{n}
\end{equation*}
\noindent by Lemma \ref{lemma2}. The proof of \eqref{EstimationsurU} is complete
since the first term is equal to $c_{n}\int_{\mathbb{T}} u_{n}^{2}(\psi _{n}(y),y)/2 \, dy\geq 0.$ 

\noindent Consequently, there is a constant $C$ independent of $n$ such that
\begin{equation}
\max_{\tilde{U}} |\tilde{u}_{n}| \leq C 
\text{ and } \int_{\tilde{U}} |\nabla \tilde u_n|^2 \leq C,
\label{EstimationsurUtilde}
\end{equation}
thanks to \eqref{Extensionun} and \eqref{A2}. Note that $\varphi_n$ is not assumed uniformly bounded outside $B_0$, but in that region we use that $\tilde{U} \cap \{y \notin B_0\} \subset  \Omega_n$ to deduce \eqref{EstimationsurUtilde} from \eqref{EstimationsurU}.

\bigskip

{\bf Step 2:} {\it Limiting problem.}

\smallskip

\noindent By \eqref{EstimationsurUtilde}, $u_n$ converges weakly  to some $\tilde u$ in $H_{\rm loc}^1(\tilde{U})$ up to a subsequence. Taking any test $\varphi \in C^\infty_c(\tilde{U})$ in \eqref{A1}, we have

\begin{equation}\label{apptech}
c_n \int_{\Omega_n} (u_n)_{x} \varphi+\int_{\Omega_n}\nabla u_n\nabla \varphi=c_n \int\limits_{B_0} \varphi(\psi_n(y),y) \, dy.
\end{equation}
The LHS can be rewritten as
$$
\int_{U \cap \{y \notin B_0\}} \left(c_n \, (u_n)_{x} \varphi+\nabla u_n\nabla \varphi\right)+\int_{\tilde{U} \cap \{y \in B_0\}} \left(c_n \, (u_n)_{x} \varphi+\nabla u_n\nabla \varphi\right) \mathbf{1}_{x<\psi_n(y)},
$$
 since $supp \, \varphi \cap \{y \notin B_0\} \subset U \cap \{y \notin B_0\}
$. 
This converges to
$$
\int_{U \cap \{y \notin B_0\}} \left(\overline{c} \, \tilde{u}_{x} \varphi+\nabla \tilde{u} \nabla \varphi\right)+\int_{\tilde{U} \cap \{y \in B_0\}} \left(\overline{c} \, \tilde{u}_{x} \varphi+\nabla \tilde{u} \nabla \varphi\right) \mathbf{1}_{x<\overline{\psi}(y)},
$$
because $\mathbf{1}_{x<\psi_n(y)} \to \mathbf{1}_{x<\overline{\psi}(y)}$ {\it a.e.} in $\tilde{U} \cap \{y \in B_0\}$ by the uniform convergence of $\psi_n$  to $\overline{\psi}$ in $B_0$, cf. \eqref{A2}. Noting that 
$$
\tilde{U} \cap \{x<\overline{\psi}(y), \, y \in B_0\}=U \cap \{y \in B_0\}, 
$$
the LHS of \eqref{apptech} thus converges to
$
\int_{U} \left(\overline{c} \, \tilde{u}_{x} \varphi+\nabla \tilde{u} \nabla \varphi\right).
$
Using again the uniform convergence of $\psi_n$ to $\overline{\psi}$ in $B_0$ to handle the RHS, we obtain that
$$
\int_{U} \left(\overline{c} \, \tilde{u}_{x} \varphi+\nabla \tilde{u} \nabla \varphi\right)=\overline{c} \int\limits_{B_0} \varphi(\overline{\psi}(y),y) \, dy \quad \forall \varphi \in C^\infty_c(\tilde{U}).
$$
This proves that $\tilde{u} \in H_{\rm loc}^1(\tilde{U})$ is a variational solution of \eqref{A5}. Note that it is nonnegative by \eqref{EstimationsurU}.

\bigskip

{\bf Step 3:} {\it Local uniform convergence.}

\smallskip

\noindent When calling for Lemma \ref{lemmaA1}, it was important that $\tilde{u}_n \to \tilde{u}$ locally uniformly in $\tilde{U}$, especially around the moving interfaces $\{x=\psi_n(y), \, y \in B_0\}$. Let us adapt an idea of \cite{Nit11} to establish this convergence. It consists in considering the problem satisfied by $\tilde{u}_n$ in order to apply standard interior elliptic estimates, cf. also \cite[Appendix A.4]{IFB}. 

\noindent By \cite[Lemma A.5]{IFB}, 
\begin{equation*}
b_{n} \, (\tilde{u}_{n})_{x}-\textrm{div}(A_{n}\nabla \tilde{u}_{n})=-(b_{n})_{x}%
\text{ in }\mathcal{D}'(\R \times \mathbb{T}),
\end{equation*}%
with
\begin{equation*}
b_{n}(x,y)=c_{n}\mathbf{1}_{\{x<\psi _{n}(y)\}}-c_{n}\mathbf{1}_{\{x>\psi
_{n}(y)\}},
\end{equation*}%
\begin{equation*}
A_{n}(x,y)=\left( 
\begin{array}{cc}
1 & 0 \\ 
0 & 1%
\end{array}%
\right) \mathbf{1}_{\{x<\psi _{n}(y)\}}-\left( 
\begin{array}{cc}
1+4\,\psi _{n y}^{2} & 2\,\psi _{n y} \\ 
2\,\psi _{n y} & 1%
\end{array}%
\right) \mathbf{1}_{\{x>\psi _{n}(y)\}}.
\end{equation*}%
We claim that there are $\Lambda \geq \lambda >0$, $\nu \geq 0$, such that for each $n$ and {\it a.e.} $x,y\in
\tilde{U},$%
\begin{equation*}
\Lambda \geq A_{n}(x,y)\geq \lambda \text{ and } \lambda
^{-1}\,|b_{n}(x,y)|\leq \nu.
\end{equation*}%
For $y \notin B_0$, this follows once again from the fact $\tilde{U} \cap \{y \notin B_0\} \subset \Omega_n$, so 
 $b_n=c_n$ and $A_n=Id$ everywhere in that region.  
For $y \in B_0$, recall that $(\psi_n)_n$ is bounded in $W^{1,\infty}(B_0)$ and the proof is the same as in \cite[Lemma A.7]{IFB}.

\noindent Now by \cite[Thm. 8.24]{GiTr01} combined with 
\eqref{EstimationsurUtilde}, $(\tilde{u}_n)_n$ is locally equi-H\"older continuous in $\tilde{U}$ and the convergence of $\tilde{u}_n$ to $\tilde{u}$ is local uniform, up to taking another subsequence if necessary. 

\bigskip

{\bf Step 4:} {\it $C^1$ regularity.}

\smallskip

\noindent This step is now standard. For example, let us call for \cite[Thm. 2.4.2.7]{Grisvard} to get the regularity around the boundary $\{x=\overline{\psi}(y), \, y \in B_0\}$.  For any arbitrary $\rho \in C^\infty_c(\tilde{U})$, consider a nonempty $C^{1,1}$ open $V$ such that $\overline{U} \cap supp \, \rho \subset \overline{V} \subset \overline{U}$ and rewrite \eqref{A5} for the function $v:=\tilde{u} \rho$ as follows:
\begin{equation*}
\begin{cases}
v-\triangle v=f \text{ in }V, \\ 
\frac{\partial v}{\partial \nu_V}=g\text{ on }
\partial V,
\end{cases}
\end{equation*}
with $f=
\tilde{u} \rho-c \rho \tilde{u}_x-2 \nabla \tilde u \nabla \rho-\tilde{u} \triangle \rho$ and $$
g=
\begin{cases}
\frac{\overline{c} \, \rho}{\sqrt{1+\overline{\psi}_y^2}}+\tilde{u} \frac{\partial \rho}{\partial \nu}  & \mbox{on } \partial V \cap \{x=\overline{\psi}(y), \, y \in B_0\},\\
0 & \mbox{elsewhere on $\partial V$,}
\end{cases}
$$
where $\nu_V$ is the outer unit normal of $\partial V$.
Such a regular open $V$ exists because $supp \, \rho \subset \tilde{U}$ and $\overline{\psi} \in W^{2,\infty}(B_0)$. Using then that $\tilde{u} \in L^\infty \cap H^1(\tilde{U})$, we have $f \in L^2(V)$ and $g \in H^{1/2}(\partial V)$. It follows that $v \in H^2(V)$ by \cite[Thm. 2.4.2.7]{Grisvard}.
Since $\rho$ is arbitrary, $\tilde{u} \in H_{\rm loc}^2(U)$ thus a fortiori in $H_{\rm loc}^2(\tilde{U})$ because $\overline{\psi} \in W^{2,\infty}(B_0)$.
By Sobolev's embeddings \cite[Cor. IX.14]{Bre83}, $\nabla \tilde{u} \in L^p_{\rm loc}(\tilde{U})$ for any $p \in [2,\infty)$, so $f \in L^p(V)$ and $g \in W^{1-1/p,p}(\partial V)$ thanks to trace results \cite[Thm. 1.5.1.3]{Grisvard}. Here we used that $\tilde{u} \in L^\infty \cap W^{1,p}_{\rm loc}(\tilde{U})$. Applying again \cite[Thm. 2.4.2.7]{Grisvard}, $v \in W^{2,p}(V)$, $\tilde{u} \in W^{2,p}_{\rm loc}(\tilde{U})$, and by \cite[Cor. IX.14]{Bre83}, $\tilde{u} \in C^1(\tilde{U})$. To get the $C^2$ regularity in $U$, use \textit{e.g.} \cite[Rem. 26]{Bre83}.
\end{proof}

\section{Technical features}\label{appmissing}

Here is another result used in the proofs.

\begin{lemma}\label{lemma7}
Let $G \in C(\R)$ and $R>0$. 
There exists $r>0$ such that for any $h\in W^{1,\infty }(\mathbb{T})$ and $y_0 \in \mathbb{T}$ such that $| h_{y}| \leq G(h)$ {\it a.e.} in $\mathbb{T}$ and $|h(y_{0})|\leq R,$ we have
\begin{equation*}
| h(y)| \leq 2R \quad \forall y \in B(y_{0},r).
\end{equation*}
\end{lemma}

\begin{proof}
Set
$
C_{2R}
:=
\max \{ \max_{[-2R,2R]} | G| ,1\} 
$, let $h=h(y)$ and $y_{0}$ be as in the lemma, and define 
$$
r _{0}
:=
\sup \{r>0; \, \sup_{B(y_0,r)} |h| \leq 2R\}.
$$
The above set is not empty because $|h(y_0)| \leq R$ and $h$ is continuous, so $r_0$ is well-defined in $(0,+\infty]$.
We claim that $r _{0}\geq R/C_{2R}$ and this will show that any $r\leq R/C_{2R}$ fits the
lemma. If indeed $r _{0}< R/C_{2R}$  then for
all $y \in B(y_{0},r_{0}),$ 
\begin{equation*}
|h(y)| \leq |h(y_{0})| +r
_{0}| h_{y}| _{L^{\infty }(B(y_0,r_0))} 
\leq R+r_{0}C_{2R} 
<2R,
\end{equation*}
but one can then choose $r _{1}>r _{0}$
such that  $\sup_{B(y_0,r_1)} |
h(y)| \leq 2R$ and this contradicts the definition of $r _{0}.$ 
\end{proof}

\noindent It only remains to check that \eqref{missing} holds, cf. Remark \ref{remark3}.

\begin{proof}[Proof of \eqref{missing}] 
The domain $\Omega_0$ is not Lipschitz at $(x_0,y_0) \in \partial \Omega_0$ whenever $\psi(y_0)=x_0$ and $\psi_y(y_0)=0$.
But the variational solution $u \in H^1(\Omega)$ of \eqref{Eqn_laplacien} is in $C^1(\overline{\Omega}) \cap C^2(\Omega)$ by standard elliptic regularity; cf. the fourth  step of the proof of Lemma \ref{lemmaA1}. The strong maximum principle \cite[Sec. 6.4.2]{Evans} then implies that $\min_{\overline{\Omega}_0} (u-w_0)$ is not achieved in $\Omega_0$ or on $\partial \Omega_0 \cap \{x>x_0\}$ where $\frac{\partial (u-w_0)}{\partial \nu}>0$ pointwise. It is then achieved on $\partial \Omega_0 \cap \{x=x_0\}$ where $u \geq w_0$ pointwise, including at $(x_0,y_0)$ such as above.
\end{proof}


\begin{thebibliography}{99}
\bibitem{IFB} N. Alibaud, G. Namah, On the propagation of a periodic flame
front by an Arrhenius kinetic,  \textit{Interfaces Free Bound.,
}  \textbf{19} (2017) no. 3, 449--494.


\bibitem{BrFiNaSc93} C. M. Brauner, P. C. Fife, G. Namah and C. Schmidt-Lain%
\'{e}, Propagation of a combustion front in a striated solid medium: a
homogenization analysis, \textit{Quart. Appl. Math.} \textbf{51} (1993), no.
3, 467--493.

\bibitem{Italiens} A. Cesaroni, M. Novaga, Long-time behavior of the mean
curvature flow with periodic forcing,  \textit{Comm. Partial Differential Equations},  \textbf{38}
 (2013),  no. 5, 780-801.

\bibitem{Bre83} H. Br\'{e}zis, \textit{Analyse fonctionnelle : Th\'{e}orie
et Applications,} Collection Math\'{e}matiques Appliqu\'{e}es pour la Ma%
\^{\i}trise, Masson, Paris, 1983.

\bibitem{ChNa97} X. Chen, G. Namah, Wave propagation under curvature effects
in a heterogeneous medium, \textit{Appl. Anal.} \textbf{64} (1997), no. 3-4,
219--233.

\bibitem{Evans} L.C. Evans, Partial Differential Equations, Graduate Studies
in Mathematics, Vol. 19,  \textit{AMS,}  1997.

\bibitem{GiTr01} D. Gilbarg, N. Trudinger, \textit{Elliptic partial
differential equations of second order,} Reprint of the 1998 edition,
Classics in Mathematics, Springer-Verlag, Berlin, 2001.

\bibitem{Grisvard} Grisvard,  \textit{Boundary Value Problems in Non Smooth Domains,} 
Pitman, London, 1985.

\bibitem{Nit11} R. Nittka, Regularity of solutions of linear second order
elliptic and parabolic boundary value problems on Lipschitz domains, \textit{%
J. Differential Equations} \textbf{251} (2011), no. 4-5, 860--880.


\end{thebibliography}
\end{document}